\documentclass[12pt,reqno,oneside,a4paper]{article}

\usepackage{amsthm}
\usepackage{amsmath}
\usepackage{float}
\usepackage{graphicx}
\usepackage{enumitem}
\usepackage{titlesec}
\usepackage{enumitem}
\usepackage{amssymb}
\usepackage[nottoc,notlot,notlof]{tocbibind}
\usepackage[utf8]{inputenc}
\usepackage{framed}
\usepackage{pgf,tikz,pgfplots}
\usepackage{mathrsfs}
\usetikzlibrary{arrows}
\usepackage{blindtext}
\usepackage{mathrsfs}
\usepackage{realboxes}
\usepackage{chngcntr}
\usepackage{relsize}
\usepackage{mathtools}
\usepackage{subcaption}
\usepackage{pgf,tikz,pgfplots}
\usepackage{mathrsfs}
\usetikzlibrary{arrows}
\newtheorem{theorem}{Theorem}
\newtheorem{lemma}[theorem]{Lemma}
\newtheorem{cor}[theorem]{Corollary}
\newtheorem{prop}[theorem]{Proposition}

\theoremstyle{definition}
\newtheorem{definition}[theorem]{Definition}

\theoremstyle{remark}

\newcommand{\pra}{\mathbb{R}}
\newcommand{\mig}{\mathbb{C}}
\newcommand{\fis}{\mathbb{N}}

\makeatletter
\newcommand{\displaybump}{\hbox to \@totalleftmargin{\hfil}}
\makeatother
\numberwithin{theorem}{section} 
\numberwithin{equation}{section}

\title{A new family of entire functions with no wandering domains}
\author{Yannis Dourekas}
 \date{}

\begin{document}
\maketitle

\begin{abstract}
The issue of whether an analytic function has wandering domains has long been of interest in complex dynamics. Sullivan proved in 1985 that rational maps do not have wandering domains. On the other hand, several transcendental entire functions have wandering domains. Using recent results on the relationship between Fatou components and the postsingular set, we find a new family of transcendental entire functions that does not have wandering domains. We also prove that the Julia set of a certain subfamily is the whole plane.
\end{abstract}
\section{Introduction}
        Let $f : \mig \to \mig$ be a transcendental entire function. We denote by $f^n$ the $n$th iterate of $f$, for $n=0, 1, 2 \ldots$. The \emph{Fatou set} of $f$, $F(f)$, is the set of points $z \in \mig$ such that the sequence $\{f^n\}_{n \in \fis}$ forms a normal family in some neighborhood of $z$. The complement of the Fatou set is the \emph{Julia set} of $f$, $J(f)$. Another set of note is the \emph{escaping set} of $f$, $I(f)$, defined as the set of points that tend to infinity under iteration.

Let $U$ be a component of the Fatou set of $f$ such that, for any $n, m \in \fis$ with $n \neq m$, $f^n(U) \cap f^m(U) = \varnothing$. Then $U$ is called a \emph{wandering domain} of~$f$.

One of the most celebrated results in complex dynamics is that rational functions do not have wandering domains. Sullivan proved this result in 1985, using tools such as quasiconformal homeomorphisms \cite{sull}.

Wandering domains do exist for transcendental entire functions. It is known that, if $U$ is a wandering domain of a transcendental entire function $f$, all limit functions of $\{f^n|_U\}$ are constant \cite[Section 28]{fatou}. Wandering domains for transcendental entire functions can thus be completely categorised into three groups:
\begin{itemize}
\item if the only limit function is $\infty$, $U$ is called escaping;
\item if the limit functions all lie in a bounded set, $U$ is called of bounded orbit; and
\item if the limit functions include both finite values and $\infty$, $U$ is called oscillating.
\end{itemize}
 Most known examples of wandering domains are escaping, and the first such example was constructed by Baker in 1976 \cite{baker1}. Another example was given by Herman \cite{herman}. The first example of an oscillating wandering domain was given by Eremenko and Lyubich in 1987 \cite{erl2}, with more recent examples being given by Bishop in 2015 \cite{bishop}, and by Martí-Pete and Shishikura in 2018 \cite{david}. Note that the existence of  wandering domains where all limit functions of $\{f^n|_U\}$ lie in a bounded set is an open question.

There are, on the other hand, several families of transcendental entire functions which have been shown not to have wandering domains, as described in  \cite[Section~4.6]{bsurvey}. The methods used to prove these results usually emanate from Sullivan's techniques. Alternative techniques have been used, for example by Bergweiler in 1993 \cite{berg} and Mihaljević-Brandt and Rempe-Gillen in 2013 \cite{lasse1}. 

In this paper we use a new technique based on results on the relationship between wandering domains and points in the postsingular set, proved in 2017 by Barański, Fagella, Jarque and Karpińska \cite[Theorem B]{nuria}. We consider the family of functions $\mathcal{F}_p$, for $p \geq 3$, defined as
\begin{align*}
\mathcal{F}_p = \left\{ f_{\lambda}: f_{\lambda}(z) = \lambda \sum_{k=0}^{p-1} \exp( \omega_p^k z) \text{, for some } \lambda \in \pra^{*} \right\},
\end{align*}
where $\omega_p= \exp(2 \pi i/p)$ is a $p$th root of unity. We use several properties of these functions from \cite{dave2} and \cite{cbsw}; most importantly, that if $f \in \mathcal{F}_p$, for some $p \geq 3$, then $J(f)$ has a structure known as a \emph{spider's web}, which is a connected set containing a sequence of loops, while also containing a \emph{Cantor bouquet} (a collection of pairwise disjoint curves to infinity that is ambiently homeomorphic to a straight brush).

Our main result is the following.
\begin{theorem}  Let $f_{\lambda} \in \mathcal{F}_p$, $p \geq 3$, with $p$ even and $\lambda \in \pra^{*}$. Then $f$ has no wandering domains. \label{theo}
\end{theorem}
Using Theorem \ref{theo}, we can also prove the following.
\begin{theorem}  Let $f_{\lambda} \in \mathcal{F}_p$, $p \geq 3$, with $p$ even and $|\lambda| \geq 1$. Then $J(f) = \mig$. \label{theo2}
\end{theorem}

We note that the methods discussed here cannot be applied in the case where $p$ is odd, as the situation is more delicate in that case, with the postsingular set lying in a part of the plane where we do not have good control of the dynamics.\\

The structure of this paper is as follows:
\begin{itemize}
\item In Section 2, we fix some notation and state relevant definitions and results from transcendental dynamics. These include the definition of a spider's web (by Rippon and Stallard; see \cite{rs1}), along with the main theorem we use to get our results, by Barański, Fagella, Jarque and Karpińska (\cite[Theorem B]{nuria}).
\item In Section 3, we discuss properties of the functions in $\mathcal{F}_p$. Many of these properties were proved by Sixsmith \cite{dave2}, while others follow from results we proved in \cite{cbsw}.
\item Section 4 contains the proofs of Theorems \ref{theo} and \ref{theo2}, along with an example that demonstrates that there exist small values of $\lambda$ for which the result of Theorem \ref{theo2} does not hold.
\end{itemize}

\emph{Acknowledgements.} I would like to thank my supervisors, Prof Gwyneth Stallard and Prof Phil Rippon, for their boundless patience and unsparing guidance.
\section{Preliminaries}
In this section we give some notation and results from transcendental dynamics that we will use in the proofs of Theorems \ref{theo} and \ref{theo2}.

 The point $z$ is a \emph{critical point} of $f$ if $f'(z) = 0$. If $z$ is a critical point, then $f(z)$ is called a \emph{critical value} of $f$. A (finite) \emph{asymptotic value} of $f$ is a value $w \in \mig$ for which there exists a curve  $\gamma: (0, \infty) \mapsto \mig$ with $\gamma(t) \to \infty$ as $t \to \infty$, such that $f(\gamma(t)) \to w$ as $t \to \infty$. We denote the set of critical points, the set of critical values, and the set of finite asymptotic values of $f$ by $CP(f)$, $CV(f)$ and $AV(f)$ respectively. We define the set of \emph{singular values} of $f$ as
\begin{equation*}
S(f) : = \overline{CV(f) \cup AV(f)}.
\end{equation*}
 We further define the \emph{postsingular set} as 
\begin{equation*}
P(f) : = \overline{ \cup_{n \geq 0} f^n(S(f))}.
\end{equation*}

The \emph{fast escaping set} of $f$, $A(f)$, is roughly defined as the set of points that tend to infinity under iteration ``as fast as possible''. The formal definition of the fast escaping set, which can be found in \cite{rs1}, along with an extensive study of many of its properties, is
\begin{align*}
A(f) = \{z \in \mig: \exists L \in \fis \text{ such that } |f^{n+L}(z)| \geq M^n(R,f) \text{ for } n \in \fis \},
\end{align*}
where
\begin{align*}
M(r,f) = \max_{|z|=r} |f(z)|, \text{ for } r>0,
\end{align*}
$M^n(r,f)$ denotes iteration of $M(r,f)$ with respect to the variable $r$, and $R>0$ is any value large enough so that $M(r,f)>r$ for $r \geq R$. We define $A_R(f) := \{ z : |f^n(z)| \geq M^n(R,f), \text{ for } n \in \fis\}$. 

In the same paper \cite{rs1}, the notion of a \emph{spider's web} is introduced. This is a connected structure containing a sequence of loops. The formal definition is as follows:
\begin{definition} \label{swdef}
A set $E$ is an (infinite) spider's web if $E$ is connected and there exists a sequence of bounded simply connected domains $G_n$, with $G_n \subset G_{n+1}$, for $n \in \fis$, $\partial G_n \subset E$, for $n \in \fis$ and $\cup_{n \in \fis} G_n = \mig$.
\end{definition}

It is known that the escaping, fast escaping, and Julia sets of many transcendental entire functions are spiders' webs  \cite{rs1}. We note that the spiders' webs that arise in complex dynamics are extremely elaborate; see \cite{john} and \cite{rs1}.

We turn our attention to wandering domains. 
We start by quoting \cite[Theorem 1.5 (a)(ii)]{osb}: 
\begin{theorem} \label{wand}
Let $f$ be a transcendental entire function, let $R>0$ be such that $M(r,f)>r$ for $r \geq R$, and let $A_R(f)$ be a spider's web. Suppose that $K$ is a component of $A(f)^{\mathsf{c}}$ with bounded orbit. Then, if the interior of $K$ is non-empty, this interior consists of non-wandering Fatou components.
\end{theorem}
An immediate corollary of this theorem, and one that we will use to prove our results, is the following:
\begin{cor} \label{osb2}
Let $f$ be a transcendental entire function, let $R>0$ be such that $M(r,f)>r$ for $r \geq R$, and let $A_R(f)$ be a spider's web. Then $f$ has no wandering domains with bounded orbit.
\end{cor}
We now quote \cite[Theorem B]{nuria} which was mentioned in the introduction as providing a new technique for ruling out wandering domains; this forms the basis for our results. It describes a relationship on the distance between the postsingular set and forward images of Fatou components.
\begin{theorem} \label{nuria} Let $f$ be a transcendental meromorphic map and $U$ be a Fatou component of $f$. Denote by $U_n$ the Fatou component such that $f^n(U) \subset U_n$. Then, for every $z \in U$, there exists a sequence $(p_n)$ in $P(f)$ such that 
\begin{align*}
\frac{\operatorname{dist}(p_n,U_n)}{\operatorname{dist}(f^n(z),\partial U_n)} \to 0, \text{ as } n \to \infty.
\end{align*}
In particular, if for some $d>0$ we have $\operatorname{dist}(f^n(z), \partial U_n)<d$ for all $n$, then $\operatorname{dist} (p_n, U_n) \to 0$ as $n$ tends to $\infty$.
\end{theorem}

\section{Properties of functions in \textbf{$\mathcal{F}_p$}}
Sixsmith studied in \cite{dave2} the class $\mathcal{E}_p$ of transcendental entire functions defined for $p \geq 3$ as
\begin{align*}
\mathcal{E}_p = \left\{ f: f(z) = \sum_{k=0}^{p-1} a_k \exp ( \omega_p^k z), \text{ where } a_p \in \mig^* \text{ for } k \in \{0, 1 ,\ldots p-1 \} \right\},
\end{align*}
where $\omega_p = \exp (2 \pi /p)$ is an $n$th root of unity. We will restrict our studies to the family $\mathcal{E}_p$, $p \geq 3$, where $a_i \in \pra$ and $a_i = a_j$ for all $ i,j \in \{0, 1 ,\ldots p-1 \}$; that is, the family $\mathcal{F}_p$ as defined in Section 1. The reason for the restriction is that we have strong control over points in $P(f)$ for $\mathcal{F}_p$ (as will be seen in results quoted in this section from \cite{cbsw}), which we lack for the larger class $\mathcal{E}_p$. This control is essential in order to apply \cite[Theorem B]{nuria}.

We will also quote some results from  \cite{dave2} in this section, sometimes modifying them for our purposes. Since $\mathcal{F}_p \subset \mathcal{E}_p$, all the results proven for $\mathcal{E}_p$ will hold for $\mathcal{F}_p$ as well. The first one of these is (\cite[Theorem 1.2]{dave2}):
\begin{theorem} \label{sw}
Suppose that $f \in \mathcal{E}_p$, $p \geq 3$. Then each of 
\begin{align*}
A_R(f), A(f), I(f), J(f) \cap A_R(f), J(f) \cap A(f), J(f)\cap I(f), \text{ and } J(f)
\end{align*}
is a spider's web, where $R>0$ is such that $M(r,f)>r$ for $r \geq R$.
\end{theorem}

 We give a partition of the plane induced by each of the functions in $\mathcal{E}_p$ as in \cite{dave2}. This partition allows for distinct dynamical properties in the different components. The partition is pictured in Figure 1. \\

\begin{figure}[h]
\centering
\begin{subfigure}[b]{.8\textwidth}
\begin{tikzpicture}
\draw [line width=0.8pt] (2.75316535371619,-1.9972747496061851)-- (2.75316535371619,2.0027252503938153);
\draw [line width=0.8pt] (2.75316535371619,2.0027252503938153)-- (-1.0510607114644228,3.2387932278936047);
\draw [line width=0.8pt] (-1.0510607114644228,3.2387932278936047)-- (-3.402201720634316,0.002725250393815548);
\draw [line width=0.8pt] (-3.402201720634316,0.002725250393815548)-- (-1.0510607114644244,-3.233342727105974);
\draw [line width=0.8pt] (-1.0510607114644244,-3.233342727105974)-- (2.75316535371619,-1.9972747496061851);
\draw [line width=0.8pt] (-6,0.2)-- (-3.2600367285857894,0.1983985751733357);
\draw [line width=0.8pt] (-3.252230792018594,-0.2036920243491368)-- (-6,-0.2);
\draw [line width=0.8pt] (-2.0417214023631414,-5.642652597187732)-- (-1.1935031420162177,-3.0372875409865143);
\draw [line width=0.8pt] (-0.808680090143069,-3.1541164826793517)-- (-1.6612987958450798,-5.766259394937711);
\draw [line width=0.8pt] (4.738146777781479,-3.683584888353489)-- (2.5224112211399743,-2.071779305693192);
\draw [line width=0.8pt] (2.752439010284849,-1.7418929630065727)-- (4.973260878698468,-3.3599780906035104);
\draw [line width=0.8pt] (4.970057154717891,3.3698381391561885)-- (2.7524390102848493,1.7606237162958618);
\draw [line width=0.8pt] (2.509780950460227,2.0813336301934253)-- (4.734943053800901,3.6934449369061664);
\draw [line width=0.8pt] (-1.6664825301362274,5.770025598354111)-- (-0.8213103608228155,3.163670807179586);
\draw [line width=0.8pt] (-1.2013090785834117,3.031994091810713)-- (-2.0469051366542876,5.646418800604131);
\draw (-0.49590073949914104,0.43137351827695) node[anchor=north west] {$P(\nu)$};
\draw (3.7789459447540286,3.2517304991653325) node[anchor=north west] {\tiny{$Q_0$}};
\draw (-1.9394925021909722,5.217118143203081) node[anchor=north west] {\tiny{$Q_4$}};
\draw (-5.295387698637927,0.19820006277223147) node[anchor=north west] {\tiny{$Q_3$}};
\draw (-1.59646124105624335,-3.740628772436522) node[anchor=north west] {\tiny{$Q_2$}};
\draw (3.6234969744175496,-2.5832946755320314) node[anchor=north west] {\tiny{$Q_1$}};
\draw (4.303586219639644,0.4953556692325308) node[anchor=north west] {$T_0(\nu)$};
\draw (0.9614333574053485,-3.876646621480941) node[anchor=north west] {$T_1(\nu)$};
\draw (1.2140379342021268,4.867357959946003) node[anchor=north west] {$T_4(\nu)$};
\draw (-4.362693876619053,3.1962815288288535) node[anchor=north west] {$T_3(\nu)$};
\draw (-4.576436210831712,-2.361019160700271) node[anchor=north west] {$T_2(\nu)$};
\end{tikzpicture}
\end{subfigure}\caption{The sets $P(\nu)$, $Q_k$ and $T_k(\nu)$, $0 \leq k \leq p-1$, for $p=5$.}
\end{figure}
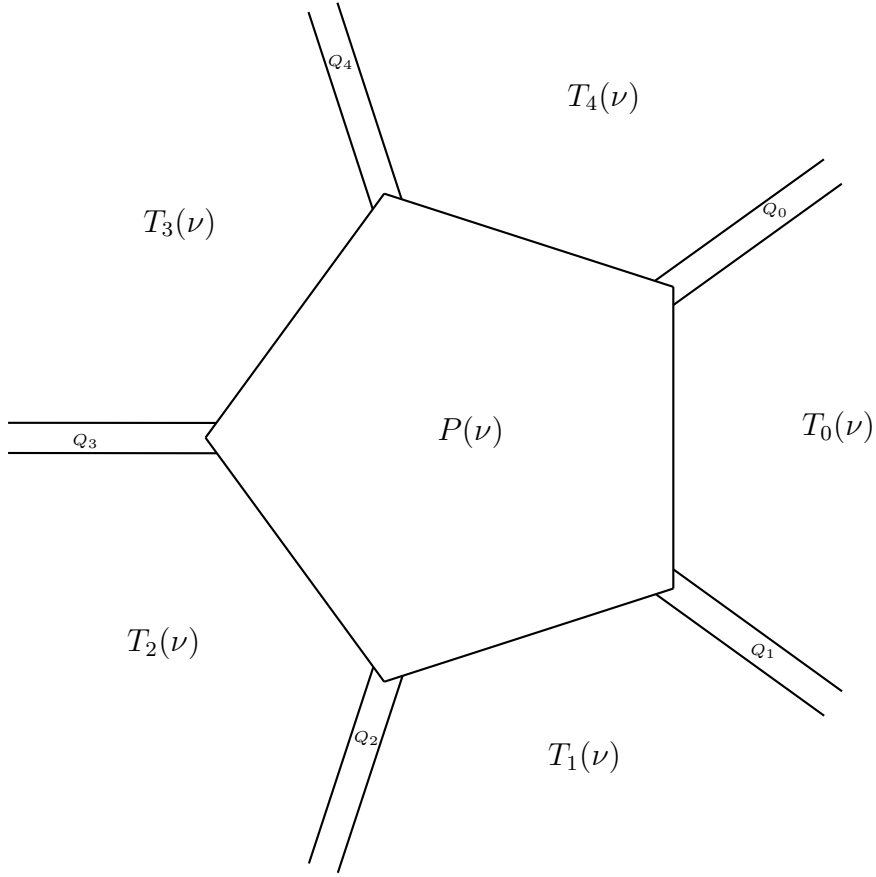
For $\nu >0$ we let $P(\nu)$ be the interior of the regular $p$-gon centered at the origin and with vertices at the points
\begin{align*}
\frac{\nu}{\cos(\pi/p)} \exp \left( \frac{(2k+1)i \pi}{p} \right), \text{ for } k \in \{0,1,\cdots,p-1\}.
\end{align*}
For some $q \geq \log (32 p)/2 \sin(\pi/p)$, define the domains
\begin{align*}
Q_k = \left\{ z \exp \left( \frac{(-2k+1)i \pi}{p} \right) : \operatorname{Re} (z) > 0, |\operatorname{Im}(z)| <  q \right\},
\end{align*}
for $k \in \{0,1,\cdots,p-1\}$.

Roughly speaking, each $Q_k$ can be obtained by rotating a half-infinite horizontal strip of width $2 q$ around the origin until a vertex of $P(\nu)$ is positioned centrally in the strip.

Set
\begin{align*}
T( \nu) = \mig \setminus \left( P(\nu) \cup \bigcup_{k=0}^{p-1} Q_k \right).
\end{align*}
The set $T(\nu)$ consists of $p$ unbounded simply connected components, which are arranged rotationally symmetrically. We label these $T_j(\nu)$, for $j \in  \{0,1,\cdots,p-1\}$, where $T_0(\nu)$ has unbounded intersection with the positive real axis, and $T_{j+1} (\nu)$ is obtained by rotating $T_j(\nu)$ clockwise around the origin by $2 \pi /p$ radians.

A key lemma \cite[Lemma 4.1]{dave2} concerns the behaviour of $f$ in each of the regions $T_j(\nu)$, $j \in \{0,1,\cdots,p-1\}$. For our purposes, we quote a small part of that lemma:
\begin{lemma} \label{davel}
Let $f \in \mathcal{E}_p$, $p \geq 3$. There exist $\nu'>0$ and $\varepsilon_0 \in (0,1)$ such that, for all $z \in T(\nu)$ with $\nu \geq \nu'$,
\begin{equation}
|f(z)| > \max\{ e^{\varepsilon_0 \nu}, M ( \varepsilon_0 |z|, f)\}. \label{eq:anisotita}
\end{equation}
\end{lemma}

We proved the following result in \cite[Lemma 2.5]{cbsw}.
\begin{lemma}\label{g} Let $f \in \mathcal{F}_p$, $p \geq 3$, and let $\nu'$ be as in Lemma \ref{davel}. Let $z \in \mig$ be such that $f^n(z) \in T(\nu')$ for all $n \geq 1$. Then $z \in J(f) \cap A(f)$.
\end{lemma}

We want to locate the zeros and the critical points of $f \in \mathcal{F}_p$, $p \geq 3$. To that end, we define rays that lie inside the strips $Q_k$.
\begin{definition}
Let $V_0:=\{x+iy \in \mig: y = \tan(\pi/p) x,  x>0\}$ and let its $2 k \pi/p$-rotations clockwise around the origin for $k=1, \ldots, p-1$ be $V_1, \ldots, V_{p-1}$. 
\end{definition}
We proved the following for these rays \cite[Theorems 3.2 and 3.6]{cbsw}.
\begin{lemma} \label{cvs} Let $f \in \mathcal{F}_p$, $p \geq 3$. The following hold.
\begin{itemize} 
\item All zeros of $f$ lie in $\cup_{k=0,\ldots,p-1} V_k$.
\item $CP(f) \subset \cup_{k=0,\ldots,p-1} V_k$, and critical points are separated in each $V_k$ from each other by the zeros of $f$.
\item  $CV(f) \subset \pra$.
\end{itemize}
\end{lemma}

The next result on the location of the postsingular set for functions in $\mathcal{F}_p$ follows easily from the results above.

\begin{cor} Let $f \in \mathcal{F}_p$, $p \geq 3$. Then $\overline{P(f)} \subset \pra$. \label{corpf}
\end{cor}
\begin{proof}
From Theorem \ref{sw}, $A_R(f)$ is a spider's web, so $f$ has no asymptotic values \cite[Theorem 1.8]{rs1}. From Lemma \ref{cvs}, all the critical values of $f$ lie in $\pra$. But $f$ is real on $\pra$, so $f^n(CV(f)) \subset \pra$ for all $n \in \fis$. Therefore $\overline{P(f)} \subset \pra$.
\end{proof}

Further, in addition to $J(f)$ being a spider's web, we proved that it actually contains a Cantor bouquet \cite[Theorem 1.2]{cbsw}. 
\begin{theorem} \label{cb}
Let $f \in \mathcal{F}_p$, $p \geq 3$. Then $J(f)$ is a spider's web that contains a Cantor bouquet.
\end{theorem}

We now prove an elementary result for the functions $f \in \mathcal{F}_p$, $p \geq 3$, for even $p$, which is a consequence of \cite[Lemma 4.1]{dave2}. 

\begin{lemma} \label{sr}
Let $f \in \mathcal{F}_p$, $p \geq 3$, where $p$ is even, and let $\nu'$ be as in Lemma \ref{davel}. There exists $r_0 >0$ such that if
\begin{align*}
S_r = \{ x+iy : |x| \geq r, |y| \leq \pi /2 p\},
\end{align*}
then, for all $r \geq r_0$,  $|f(z)|>|z|$ for all $z \in S_r$ and $f(S_r) \subset T_0(\nu')$.
\end{lemma}
\begin{proof}
Let $z \in S_r$. From  \cite[Lemma 4.1]{dave2}, there exists $\epsilon(r)$ with $\epsilon (r) \to 0$ as $r \to \infty$, such that $f(z) \in B ( e^z, \epsilon(r))$; observe that this implies that $|f(z)|>|z|$ for all $z \in S_r$ for large enough $r$.

 Recall that $T_0(\nu')$ is part of the angle $\{ t e^{i \phi} : t>0, |\phi| \leq 2 \pi/p \}$. But, for sufficiently large $r$, if $z \in S_r$, then $|\operatorname{arg} e^z| \leq \pi/ 2p$, so we have $f(z) \in  B ( e^z, \epsilon(r)) \subset T_0(\nu')$. Thus $f(S_r) \subset T_0(\nu')$ for all large $r$.
\end{proof}

We define the following strips of width $2 \pi$ for all $k \in \mathbb{Z}$:
\begin{equation*}
R(k): =\{ z \in \mig: (2k-1) \pi < \operatorname{Im} z < (2k+1)\pi \},
\end{equation*}
and note the following result on the structure of the Julia set \cite[Corollary 5.7]{cbsw}.

\begin{lemma} \label{curveswd}
Let $f \in \mathcal{F}_p$, $p \geq 3$ and let $\nu \geq \nu'$, where $\nu'$ is as in Lemma \ref{davel}. For all large enough $k \in \fis$, there exist two simple unbounded curves $\gamma_k$ and $\gamma_{-k}$ in $J(f)$, with their endpoints in $Q_0$ and $Q_1$ respectively, that lie entirely inside the strips $R(k)$ and $R(-k)$ respectively, and extend to infinity through $T_0(\nu)$.
\end{lemma}

Finally, we prove a lemma which allows us to use Theorem \ref{nuria} for any function $f \in \mathcal{F}_p$, $p \geq 3$, with ease.
\begin{lemma} \label{bound2} Let $f \in \mathcal{F}_p$, $p \geq 3$. Then there exists $c >0$ such that $\operatorname{dist} (z, J(f)) \leq c$ for all $z \in F(f)$.
\end{lemma}
\begin{proof} Fix $\nu \geq \nu'$. 
Due to symmetry, it suffices to prove the result for points in the Fatou set that lie in $P( \nu) \cup T_0 (\nu) \cup Q_0 \cup Q_1$. 

Suppose $z \in P(\nu)$. Since $P(\nu)$ is a bounded region around the origin, the result holds as the Julia set is non-empty.

On the other hand, if $z \in Q_0 \cup Q_1 \cup T_0(\nu)$, then the result follows immediately from Lemma \ref{curveswd}.
\end{proof}

\section{Proofs of Theorem \ref{theo} and Theorem \ref{theo2}}
We are now ready to prove our two results.
\begin{proof}[Proof of Theorem \ref{theo}]
Let $f \in \mathcal{F}_p$ for some $p \geq 3$ with $p$ even. Fix $\nu = \nu'$, where $\nu'$ is as given in Lemma \ref{davel}. Suppose that $U$ is a wandering domain for $f$, and put $U_n = f^n(U)$ for $n \geq 0$. Since $U$ is bounded (as $J(f)$ is a spider's web from Theorem \ref{sw}), it follows that each $U_n$ is a Fatou component.

Let $z \in U$. It follows from Lemma \ref{bound2} that there exists $c>0$ such that $\operatorname{dist} (f^n(z) , \partial U_n) \leq c$ for all $n \in \fis$. It then follows from Theorem \ref{nuria} that there exists a sequence $(p_n)$ in $P(f)$, such that 
\begin{equation*} 
\operatorname{dist}(p_n, U_n) \to 0 \quad \text{as} \quad  n \to \infty.
\end{equation*}
 From Corollary \ref{corpf}, we have $P(f) \subset \pra$, so
\begin{equation} 
\operatorname{dist}(U_n, \pra) \to 0 \quad \text{as} \quad  n \to \infty. \label{eq:pra}
\end{equation}

We will show that the above properties imply that, for all large enough $n \in \fis$, $U_n$ has to lie in $T_0(\nu)$; thus giving a contradiction to Lemma~\ref{g}.

\begin{figure}[h]
  \label{fig:seta}
\includegraphics[scale=0.19]{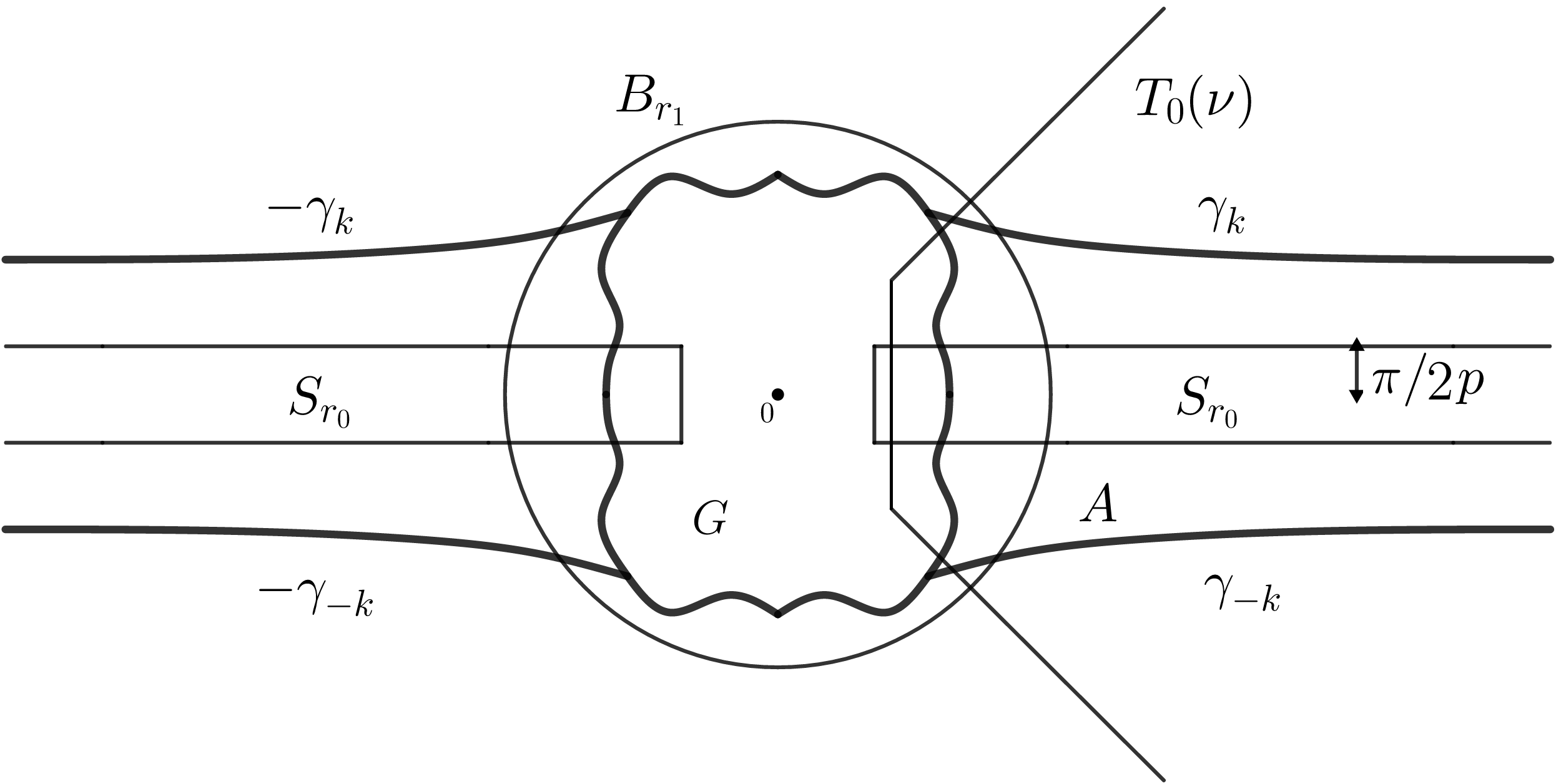}
\caption{Some of the sets described in the proof of Theorem \ref{theo}. The bold curves are in the Julia set.}
\end{figure}

To that end: from Definition \ref{swdef}, since $J(f)$ is a spider's web, we can find a domain $G$ such that the following hold (as illustrated in Figure 2):
\begin{itemize}
\item $\partial G \subset J(f)$
\item $\partial G \cap \gamma_k \neq \varnothing$ and $\partial G \cap \gamma_{-k} \neq \varnothing$,
\item $G \supset \{ z: |x| = r_0, |y| \leq \pi /2 p\}$, and
\item $P(\nu) \subset G$,
\end{itemize}
where $\gamma_k$ and $\gamma_{-k}$ are the curves in $J(f)$ defined in Lemma \ref{curveswd} for some $k \in \fis$, and $r_0$ is large enough so that  the result of Lemma \ref{sr} holds. We also define the curves $-\gamma_{k}$ and $-\gamma_{-k}$ to be the reflections about the $y$-axis of $\gamma_{k}$ and $\gamma_{-k}$ respectively. They also lie in $J(f)$ as $f$ is even.

We define $A$ to be the unbounded region in $T_0(\nu)$ with $\partial A \subset \partial G \cup \gamma_{k} \cup \gamma_{-k}\subset J(f)$ and $A \cap \pra \neq \varnothing$ (see Figure 2). 

Let $B_{r_1}$ denote the disk around the origin of radius $r_1$, where $r_1$ is large enough so that $G \subset B_{r_1}$ and
\begin{equation}
|f(z)| > |z| \text{ for } z \in T_0(\nu) \cap B_{r_1}^{\mathsf{c}}; \label{eq:display0}
\end{equation}
this is possible by \eqref{eq:anisotita}.

It follows from Theorem \ref{sw} and Corollary \ref{osb2} that points in $U$ do not have bounded orbit. Also, it follows from \eqref{eq:pra} that there exists $n_0 \in \fis$ such that
\begin{equation} \label{eq:apost}
\operatorname{dist} (U_n, \pra) \leq \pi/p \text{ for all } n \geq n_0.
\end{equation}
 Let $n_0$ also be such that $U_{n_0}$ lies outside $B_{r_1}$; this is possible because $U_n$ has to lie between two consecutive loops of the spider's web for each $n \in \fis$ and $U$ does not have bounded orbit. Thus, by Lemma \ref{sr}, we can take $w_0 \in S_{r_0} \cap U_{n_0} \cap B_{r_1}^\mathsf{c}$ such that $f(w_0) \in T_0 (\nu')$ and $|f(w_0)| > |w_0|$; in particular, $f(w_0)$ is also outside $B_{r_1}$. So
\begin{equation}
f(w_0) \in T_0(\nu) \cap B_{r_1}^{\mathsf{c}} \cap U_{n_0+1}. \label{eq:display}
\end{equation}

We know from \eqref{eq:apost} that $U_{n_0+1} \cap \{z : |\operatorname{Im} z| \leq \pi / p\} \neq \varnothing$, and, since $\partial G \cup \pm \gamma_k \cup \pm \gamma_{-k} \subset J(f)$ and $f(w_0) \in U_{n_0+1}$, it follows from \eqref{eq:display} that
\begin{equation*}
 U_{n_0 +1} \cap S_{r_0} \cap \{z: \operatorname{Re}z > 0 \} \neq \varnothing.
\end{equation*}

We thus have 
\begin{enumerate}[label=(\roman*)]
\item $U_{n_0+1}  \cap B_{r_1}^{\mathsf{c}} \neq \varnothing$
\item $U_{n_0+1} \cap S_{r_0} \cap \{z: \operatorname{Re}z > 0 \} \neq \varnothing$.
\end{enumerate}

These two properties together imply that $U_{n_0+1} \subset A \subset T_0(\nu)$. We now show that if properties (i) and (ii) are satisfied for $U_m$ (for some $m \geq n_0+1$), they will also be satisfied for $U_{m+1}$. Suppose then that (i) and (ii) are satisfied for $U_m$, for some $m \geq n_0+1$; we have $z_1 \in U_m \cap B_{r_1}^{\mathsf{c}}$ and $z_2 \in U_m \cap S_{r_0} \cap \{z: \operatorname{Re}z > 0 \}$ (and hence $U_m \subset A \subset T_0(\nu)$).

Property (i) is immediately satisfied for $U_{m+1}$, as $f(z_1) \in U_{m+1} \cap B_{r_1}^{\mathsf{c}}$ by \eqref{eq:display0}. 

Since $f(z_1) \in B_{r_1}^{\mathsf{c}}$, we have $U_{m+1} \subset \mig \setminus G$. By \eqref{eq:apost}, there exists $z_3 \in U_{m+1} \cap \{z : |\operatorname{Im}z| \leq \pi /p\}$. Since  $U_{m+1} \subset \mig \setminus G$, we have $z_3 \in S_{r_0}$. Since $f(z_2) \in U_{m+1} \cap T_0(\nu)$ by Lemma \ref{sr}, we have $z_3 \in S_{r_0} \cap \{z: \operatorname{Re}z > 0 \}$ (using the same reasoning as above), thus satisfying property (ii).

Since properties (i) and (ii) together imply that $U_m \subset A \subset T_0(\nu)$, it follows by induction that $U_m \subset A \subset T_0(\nu)$ for all $m \geq n_0+1$, giving a contradiction to Lemma \ref{g}. Thus our supposition that $U$ is a wandering domain was false.
\end{proof}

We now prove that, if $p$ is even and $|\lambda| \geq 1$, the Julia set of $f=f_{\lambda} \in \mathcal{F}_p$ is the whole plane. 

\begin{proof}[Proof of Theorem \ref{theo2}]
We offer the proof for $\lambda \geq 1$. The proof for $\lambda \leq -1$ is similar since $f$ is even as a function when $p$ is even. Let $U$ be a Fatou component of $f$. From Theorem \ref{theo}, $U$ cannot be a wandering domain. Without loss of generality, for the rest of the proof we assume that $U$ is periodic (otherwise we could work with $f^j(U)$ for some $j>0$). 

From Theorem \ref{sw} we have that $J(f)$ is a spider's web, so $U$ is bounded and, in particular, cannot be a Baker domain. The remaining cases are that $U$ belongs to an attracting cycle, a parabolic cycle, or is a Siegel disk. In each of these cases, we have $\overline{U} \cap \overline{P(f)} \neq \varnothing$ \cite[Theorem 7]{bsurvey}.  From Corollary \ref{corpf}, we have $\overline{P(f)} \subset \pra$.

 For large $x>0$ we have $f(x)>0$, and since from Lemma \ref{cvs} there are no zeros of $f$ in $\pra$ (as they lie on the rays $V_0, \ldots, V_{p-1}$ and $f(0)=\lambda p$), we have $f(x)>0$ for all $x \in \pra$. Further, we claim that $f(x)>x$ for all $x \in \pra$. Since $p$ is even, we just need to prove this for $x > 0$.

Suppose that $\lambda = 1$, which makes the function values the smallest possible within our range of values of $\lambda$. From \cite[Theorem 3.2]{cbsw}, for $x > 0$ we have 
\begin{align*}
f(x) = p \left( 1 + \frac{x^p}{p!} + \frac{x^{2p}}{(2p)!} + \ldots \right),
\end{align*} 
and thus
\begin{align*}
\frac{f(x)}{x} = \frac{p}{x} \left( 1 +  \frac{x^p}{p!}  + \frac{x^{2p}}{(2p)!} +\ldots \right).
\end{align*} 
For $x > 0$, we define
\begin{align*}
g(x) =  \frac{p}{x} \left( 1 +  \frac{x^p}{p!} \right ) <  \frac{f(x)}{x},
\end{align*} 
so we just need to show that $g(x)>1$ for $x > 0$. We have
\begin{align*}
g'(x) = - \frac{p}{x^2} + \frac{x^{p-2}}{(p-2)!} = 0
\end{align*} 
if and only if
\begin{align*}
x^p = p(p-2)!.
\end{align*} 
Hence, $g$ has a unique minimum on $(0, \infty)$, with value
\begin{align*}
\frac{p}{(p(p-2)!)^{1/p}} \left( 1 + \frac{p(p-2)!}{p!}\right) = \frac{p}{(p(p-2)!)^{1/p}} \left( 1 + \frac{1}{p-1} \right) > 1,
\end{align*} 
since $p^p > p(p-2)!$.

Thus, since $f$ is even, $f(x) > |x|$ for $x \in \pra$, and so $\pra \subset I(f)$. Since $\overline{P(f)} \subset \pra$, we have $\overline{P(f)} \subset I(f)$. Thus, $\overline{U} \cap I(f) \neq \varnothing$, which is a contradiction since $U$ is a bounded periodic Fatou component.
\end{proof}

Even though $|\lambda| \geq 1$ is not the sharpest value for the result of Theorem~\ref{theo2} to hold, in the following proposition we demonstrate that there do exist small values of $\lambda$ for which the result does not hold.

\begin{prop} Let $f=f_{1/4} \in \mathcal{F}_4$. There exists an attracting Fatou basin for $f$.
\end{prop}
\begin{proof}
For $\lambda=1/4$ and $p=4$ we can write
\begin{align*}
f(z) = \frac{1}{2} ( \cos z + \cosh z ).
\end{align*} 
We have $f(0) = 1$, while
\begin{align*}
f(\pi/2) = \frac{1}{2} (0 + \cosh(\pi/2)) \approx 1.25,
\end{align*} 
so $f(\pi/2) - \pi/2 < 0$. 

Since, for $x \geq 0$, we have
\begin{align*}
f(x) &=   1 + \frac{x^4}{4!} + \frac{x^{8}}{8!} + \ldots ,
\end{align*} 
the function $f$ is real, increasing and convex on $[0, \infty)$. We thus deduce that there exists an attracting fixed point of $f$ in $(0, \pi/2)$, which has to lie in an attracting Fatou basin.
\end{proof}

\end{document}